\documentclass[11pt]{amsart}
\usepackage[T1]{fontenc}
\usepackage[english]{babel}
\usepackage{amscd,amsmath,amsthm,amssymb,graphics}
\usepackage{lmodern,pst-node}
\usepackage{pstcol,pst-plot,pst-3d}
\usepackage{multicol}
\usepackage{epic,eepic}
\usepackage{amsfonts,amssymb,amscd,amsmath,enumitem,verbatim}
\psset{unit=0.7cm,linewidth=0.8pt,arrowsize=2.5pt 4}

\newpsstyle{fatline}{linewidth=1.5pt}
\newpsstyle{fyp}{fillstyle=solid,fillcolor=verylight}
\definecolor{verylight}{gray}{0.97}
\definecolor{light}{gray}{0.9}
\definecolor{medium}{gray}{0.85}
\definecolor{dark}{gray}{0.6}

\def\NZQ{\Bbb}               
\def\NN{{\NZQ N}}
\def\QQ{{\NZQ Q}}
\def\ZZ{{\NZQ Z}}

\def\FF{{\NZQ F}}

\def\frk{\frak}               

\def\mm{{\frk m}}
\def\nn{{\frk n}}

\def\Phi{{\frk n}}
\def\Phi{{\frk N}}
\def\MR{{\mathcal R}}

\def\MR{{\mathcal R}}

\def\xb{{\bold x}}

\def\opn#1#2{\def#1{\operatorname{#2}}} 
\opn\chara{char} \opn\length{\ell} \opn\pd{pd} \opn\rk{rk}
\opn\projdim{proj\,dim} \opn\injdim{inj\,dim} \opn\rank{rank}
\opn\depth{depth} \opn\grade{grade} \opn\height{height}
\opn\embdim{emb\,dim} \opn\codim{codim}

\opn\Tr{Tr} \opn\bigrank{big\,rank}
\opn\superheight{superheight}\opn\lcm{lcm}
\opn\trdeg{tr\,deg}
\opn\reg{reg} \opn\lreg{lreg} \opn\ini{in} \opn\lpd{lpd}
\opn\size{size}\opn\bigsize{bigsize}
\opn\cosize{cosize}\opn\bigcosize{bigcosize}
\opn\sdepth{sdepth}\opn\sreg{sreg}
\opn\link{link}\opn\fdepth{fdepth}
\opn\deg{deg}
\opn\max{max}
\opn\indeg{indeg}
\opn\div{div} \opn\Div{Div} \opn\cl{cl} \opn\Cl{Cl}
\let\epsilon\varepsilon
\let\phi=\varphi
\let\kappa=\varkappa

\opn\Spec{Spec} \opn\Supp{Supp} \opn\supp{supp} \opn\Sing{Sing}
\opn\Ass{Ass} \opn\Min{Min}\opn\Mon{Mon} \opn\dstab{dstab} \opn\astab{astab}
\opn\Syz{Syz}
\opn\Ann{Ann} \opn\Rad{Rad} \opn\Soc{Soc}
\opn\Im{Im} \opn\Ker{Ker} \opn\Coker{Coker} \opn\Am{Am}
\opn\Hom{Hom} \opn\Tor{Tor} \opn\Ext{Ext} \opn\End{End}
\opn\Aut{Aut} \opn\id{id}

\opn\nat{nat}
\opn\pff{pf}
\opn\Pf{Pf} \opn\GL{GL} \opn\SL{SL} \opn\mod{mod} \opn\ord{ord}
\opn\Gin{Gin} \opn\Hilb{Hilb}\opn\sort{sort}
\opn\initial{init}
\opn\ende{end}
\opn\height{height}
\opn\type{type}
\opn\ldim{ldim}
\opn\aff{aff} \opn\con{conv} \opn\relint{relint} \opn\st{st}
\opn\lk{lk} \opn\cn{cn} \opn\core{core} \opn\vol{vol}
\opn\link{link} \opn\star{star}\opn\lex{lex}
\opn\gr{gr}

\def\pot#1#2{#1[\kern-0.28ex[#2]\kern-0.28ex]}

\opn\dirlim{\underrightarrow{\lim}}
\opn\inivlim{\underleftarrow{\lim}}
\let\tensor=\otimes
\let\iso=\cong

\let\Dirsum=\bigoplus

\let\to=\rightarrow

\def\Implies{\ifmmode\Longrightarrow \else
        \unskip${}\Longrightarrow{}$\ignorespaces\fi}
\def\implies{\ifmmode\Rightarrow \else
        \unskip${}\Rightarrow{}$\ignorespaces\fi}
\def\iff{\ifmmode\Longleftrightarrow \else
        \unskip${}\Longleftrightarrow{}$\ignorespaces\fi}

\let\:=\colon
 \theoremstyle{plain}
\newtheorem{Theorem}{Theorem}[section]
 \newtheorem{Lemma}[Theorem]{Lemma}
 \newtheorem{Corollary}[Theorem]{Corollary}
 \newtheorem{Proposition}[Theorem]{Proposition}

 \theoremstyle{definition}
 
 \newtheorem{Remark}[Theorem]{Remark}
 
 \newtheorem{Example}[Theorem]{Example}

\let\epsilon\varepsilon
\let\kappa=\varkappa
\def\qed{\ifhmode\textqed\fi
      \ifmmode\ifinner\quad\qedsymbol\else\dispqed\fi\fi}
\def\textqed{\unskip\nobreak\penalty50
       \hskip2em\hbox{}\nobreak\hfil\qedsymbol
       \parfillskip=0pt \finalhyphendemerits=0}
\def\dispqed{\rlap{\qquad\qedsymbol}}

\opn\dis{dis}
\def\pnt{{\raise0.5mm\hbox{\large\bf.}}}

\opn\Lex{Lex}

\begin{document}
\title{Higher iterated  Hilbert coefficients of  the graded components of bigraded modules}
\author {Seyed Shahab Arkian}

\address{Seyed Shahab Arkian, University of Kurdistan, Pasdaran ST.
P.O. Box: 416, Sanandaj, Iran}

\email{Shahab\_Arkian@yahoo.com}

\begin{abstract}
Let $S=K[x_1,\ldots,x_n]$ be the polynomial ring over the field $K$, and let $I\subset S$ be a graded ideal. It is shown that the higher iterated  Hilbert coefficients of the graded $S$-modules $\Tor_i^S(M,I^k)$ and $\Ext^i_S(M,I^k)$ are polynomial functions in $k$, and an upper bound for their degree is given. These results are derived by considering suitable bigraded modules.
\end{abstract}
\thanks{}
\subjclass[2010]{Primary 13C15, 05E40, 05E45; Secondary 13D02.}
\keywords{Higher iterated  Hilbert coefficients, bigraded modules, extension functor, torsion functor }

\maketitle
\section*{Introduction}
The present paper is motivated by Kodiyalam's work  \cite{K1}, the papers by Theodorescu~\cite{T},  by Katz and Theodorescu \cite{KT1}, \cite{KT2} and the paper \cite{CKST}. In these papers it was shown  that for finitely generated $R$-modules $M$ and $N$ over a  Noetherian (local) ring $R$, and for an ideal $I\subset R$  such that the length of $\Tor_i^R(M,N/I^kN)$ is finite for all $k$, it follows that the length of $\Tor^i_R(M,N/I^kN)$  and is eventually a polynomial function in $k$. In these papers  bounds are given for the degree of these polynomials. In some cases also the leading coefficient is determined.  Similar results have been proved  for the $\Ext$-modules.

In this paper we consider a related problem. Here $I\subset S$ is  graded ideal and $S$ is the polynomial ring. It is shown in Corollary~\ref{dimension} that for any  finitely generated graded $S$-module $M$, the modules $\Tor_i^S(M,I^k)$ are finitely graded $S$-modules which for $k\gg 0$ have constant Krull dimension, and  furthermore in Corollary~\ref{ik} it is shown that the higher iterated    Hilbert coefficients (which appear  as the coefficients of the higher iterated  Hilbert polynomials) are all polynomials functions. A related result has been shown in \cite{HPV}  for the  case $M/I^kM$ and in \cite{HW} for the case $\Tor_i^s(S/\mm, I^k)$, where $\mm$ denotes the graded maximal ideal of $S$.

Observe  that knowing all higher iterated Hilbert coefficients of a graded module is equivalent to knowing its $h$-vector, and hence the Hilbert series of the module. This is the reason why we are not only interested in the ordinary Hilbert coefficients, but in all higher iterated Hilbert coefficients.

For the proof we use a technique which was first introduced by Kodiyalam \cite{K2}. For this purpose we consider the bigraded $K$-algebra $A=K[x_1,\ldots,x_n,y_1,\ldots,y_m]$ with $\deg x_i=(1,0)$ and $\deg y_j=(p_j,1)$ for all $i$ and $j$, and  a finitely generated  bigraded $A$-module $M$.   A typical example of such an $A$-module is the Rees algebra of a graded ideal $I\subset S$ with $I=(f_1,\ldots,f_m)$ and $\deg f_j=p_j$ for all $j$. For each $k$, the $S$-module $M_k=\Dirsum_i M_{(i,k)}$ is a finitely generated graded $S$-module. A graded free $S$-resolution of $M_k$ can be obtained by the graded components of the bigraded free $A$-module resolution of $M$. These resolutions are then used to compute the higher iterated Hilbert polynomials of the graded $S$-modules $M$.

The first (and important step) is to show that  the higher iterated  Hilbert coefficients of the components $A(-a,-b)_k$ of the bi-shifted free $A$-module $A(-a,-b)$ are  polynomial functions in $k$ for $k\gg 0$, see Proposition~\ref{rain}.  This result and the bigraded resolution $\FF$ of the bigraded $A$-module $M$ is then used in the next section to prove the same result for $M$. There, by using a graded version of the Noether Normalization Theorem, one obtains in Theorem~\ref{now main} and Theorem~\ref{main} upper bounds for the degree of these polynomials. The better bound for the degree of the polynomial function representing the Hilbert coefficient  $e^i_j(M_k)$ is achieved when all $p_t$ are the same and it is given by $\deg e^i_j(M_k)\leq  \dim M/\mm M+j-1$. These results are then  applied in Section~\ref{3} to show that for any finitely generated graded $S$-module $M$, and any finitely generated  bigraded module $N$, the higher iterated  Hilbert coefficients of the graded $S$-modules $\Tor_i^S(M,N_k)$ and $\Ext^i_s(M,N_k)$ are polynomial functions in $k$  for $k\gg 0$.

\section{The graded components of a bigraded module and their  higher iterated Hilbert coefficients}
\label{1}
Let $K$ be a field, $S=K[ x_{1},\ldots,x_{n}]$ the polynomial ring in $n$ variables with  the standard  grading. Let $A=K[ x_{1},\ldots,x_{n},y_{1},\ldots,y_{m}]$ with bigrading defined by $\deg x_{i} = (1,0)$ and $\deg y_{j}=(p_j,1)$, for some some integers $p_j\geq 0$.

For a finitely generated bigraded $A$-module $M= \Dirsum_{i,j\in\ZZ} M_{(i,j)}$, we define $M_{k}$ to be the graded
$S$-module $ \Dirsum_{i\in \ZZ}M_{(i,k)}$.

These definitions are motivated by the following important class of examples: Let $I\subset S$ be a graded ideal generated
by the homogeneous polynomials $f_1,\ldots,f_m$ with $\deg f_j=p_j$. Then the Rees ring $\MR(I)=\Dirsum_{k\geq 0}I^k$ is  bigraded $A$-module with $\MR(I)_k=I^k$ for all $k$.

For $ a,b\in \ZZ$,  the twisted module $A$-module  $ M(-a,-b)$ is defined to be the bigraded $A$-module with components  $M(-a,-b)_{(i,j)}=M(i-a,j-b)$.

In this section we   want to compute the Hilbert coefficients of the $S$-module $A(-a, -b)_k$   as a function of $k$.

Note that
\begin{equation}
\label{soonlunch}
A(-a, -b)_k = (A_{k-b})(-a) \iso \Dirsum_{\atop\beta_{1}+\cdots+\beta_{m}=k-b}S(-(p_1\beta_1+\cdots+p_m\beta_m)-a)y_1^{\beta_1}\cdots y_m^{\beta_m}.
\end{equation}
Hence, in a first step, we have to determine the  Hilbert  coefficients of $S(-c)$ for some $c\in \ZZ$.

Recall that for a finite graded $ S $-module $ M $ and all $ k\gg0 $, the numerical function $H(M,k)= \dim_K M_{k}$ is called the Hilbert function of $M$. For $ i\in \NN $, the higher iterated Hilbert functions $H_{i}(M,k)$ are defined recursively as follows:
\[
   H_{0}(M,k)= H(M,k), \quad  \text{and}   \quad  H_{i}(M,k)=\sum_{j\leq k}H_{i-1}(M,j).
\]
By Hilbert it is known that $H_{i}(M,k)$ is of polynomial type of degree $d+i-1$, where $d$ is the Krull dimension of $M$. In other words,  there exists a polynomial $P^i_M(x)\in\QQ[x]$ of degree $ d+i-1 $ such that  $H_{i}(M,k)=P^i_M(k)$ for all $k\gg 0$. This unique polynomial is called the {\em $i$th Hilbert polynomial} of $M$. It can be written in the form
\[
P^{i}_{M}(x)=\sum^{d+i-1}_{j=0}(-1)^{j}e^{i}_{j}(M)\binom{x+d+i-j-1}{d+i-j-1}
\]
with integer coefficients $e^i_j(M) $, called the {\em higher iterated  Hilbert coefficients} of $M$, where by definition
\[
\binom{i}{j}=\frac{i(i-1)\cdots(i-j+1)}{j(j-1)\cdots 2\cdot 1}\quad  \text{if} \quad j>0 \quad  \text{and} \quad \binom{i}{0}=1.
\]

\medskip
In  the important special case when $M=S$ we have
\begin{eqnarray*}
\label{koeln}
 P^{i}_{S}(x)=\binom{x+n+i-1}{n+i-1}.
 \end{eqnarray*}

More generally, if $c\in \ZZ$, then
\begin{eqnarray}
\label{kiana}
 P^{i}_{S(-c)}(x)={x-c+n+i-1 \choose n+i-1}
\end{eqnarray}
In particular,  $\deg P^{i}_{S(-c)}(x)=n+i-1$.

\medskip
In order to compute the higher iterated Hilbert coefficients of $M$, we define the {\em difference operator} $ \Delta $ on the set of polynomial functions by setting $(\Delta P)(a)=P(a)-P(a-1) $ for all $ a\in \ZZ$. The $d$ times iterated
$\Delta$ operator will be denoted by $ \Delta^{d} $. We further set $ \Delta^{0}P=P $.

\medskip
For our further considerations we shall need the following easy lemma  whose proof we omit.

\begin{Lemma}
Let $ P(x)=\sum_{i=0}^{n}(-1)^{i}f_{i}\binom{x+n-i}{n-i}$. Then
\begin{eqnarray*}
\label{delta}
 (\Delta^{j}P)(-1)=(-1)^{n-j}f_{n-j}, \quad    \text{for}  \quad j=0,1,...,n.
 \end{eqnarray*}
\end{Lemma}

Applying this formula  to the higher iterated  Hilbert  polynomials of $M$ we obtain
\begin{eqnarray}
\label{eij}
e^i_j(M)=(-1)^{j}\Delta^{d+i-j-1}P^i_M(-1) \quad \text{for} \quad j=0,\ldots, d+i-1,
\end{eqnarray}
where $d=\dim M$.

\medskip
Since $\Delta P_M^i=P_M^{i-1}$  for all $i \geq 1$, formula (\ref{eij}) yields

\begin{Corollary}
\label{reject}  $e^i_j(M)=e^{i-1}_j(M)$ for $j=0,\ldots,d+i-2$, where $d=\dim M$.
\end{Corollary}

Having in mind Corollary~\ref{reject},  we set $e_j(M)=e^0_j(M)$ for $j=0,\ldots,d-1$, and $e_j(M)=e^{j-(d-1)}_j(M)$ for $j\geq d$. Then for all $i$ it follows that  $e^i_j(M)=e_j(M)$  for $j=0,\ldots,i+d-1$.  Therefore,
\[
P^{i}_{M}(x)=\sum^{d+i-1}_{j=0}(-1)^{j}e_j(M)\binom{x+d+i-j-1}{d+i-j-1}.
\]
Let $M$ be a finitely generated  graded $S$-module of dimension $d$, generated in non-negative degrees, and let $H_M(t)$ be the Hilbert series of $M$. Then  there exists a unique polynomial $Q_M(t)\in \QQ [t]$  such that
$ H_M(t)=Q_M(t)/(1-t)^d.$.  Let $Q_M(t)=\sum_{i=0}^sh_it^i$ with $h_s\neq 0$. The coefficient vector $(h_0,\ldots,h_s)$ of $Q_M(t)$ is called the $h$-vector of M. The following relation between the iterated Hilbert coefficients and the $h$-vector of $M$ is well known.

\medskip
\begin{enumerate}
\item[(i)] $e_j=\sum_{i=j}^s \binom{i}{j}h_i$ for $j=0,\ldots,s$ and $e_j=0$ for $j>s$.
\item[ (ii)] $h_i=\sum_{j=i}^s(-1)^{j-i} \binom{j}{i}e_j$ for $i=1,\ldots,s$.
\end{enumerate}

These relations show that the set of higher iterated Hilbert coefficients determine the Hilbert series of $M$ completely.

\begin{Proposition}
\label{noweasy}
 Let $c\in \ZZ$. Then
the higher iterated  Hilbert  coefficients of $S(-c)$ are
\[
e^i_j(S(-c))=\binom{c}{j}\quad  \text{for all $i\geq 0$  and all $j$ with  $0\leq j\leq   n+i-1$}.
 \]
In particular, $e^i_j(S(-c))=0$ if and only if $0\leq c<j\leq n+i-1$.
\end{Proposition}
\begin{proof}
We have $\Delta^{j} P^{i}_{S(-c)}(x)=\binom{x-c+n+i-j-1}{n+i-j-1},$
and hence by formula (\ref{eij}) we get
 \[
e^i_j(S(-c))= (-1)^j\binom{j-c-1}{j}=\binom{c}{j}.
\]
\end{proof}

Now by using Proposition~\ref{noweasy}  we can give an upper bound for  the higher iterated  Hilbert coefficients of  $A(-a,-b)_{k}$.
 Before that we need an elementary lemma.
\begin{Lemma}
\label{skype}
Let $P(x)\in Q[x]$ be a polynomial of degree $d$ . Then $F(k)=\sum_{j=0}^kP(j)$ is a polynomial  in $k$ of degree $d+1$.
\end{Lemma}
\begin{proof}
Let $P(x)=\sum_{i=0}^da_ix^i$. Then $F(k)=\sum_{j=0}^k\sum_{i=0}^da_ij^i=\sum_{i=0}^da_i(\sum_{j=0}^kj^i)$. It is well-known that  $\sum_{j=0}^kj^i$ is a polynomial in $k$ of degree $i+1$.  So $F(k)$ is a polynomial  in $k$ of degree $d+1$.
\end{proof}
In proof of the  next proposition we use Lemma~\ref{skype} and also the fact that
for all $a,c,j$ we have
\begin{eqnarray*}
\binom{c+a}{j}=\sum_{i=0}^{a}\binom{a}{i}\binom{c}{j-i}.
\end{eqnarray*}

\begin{Proposition}
\label{rain}
For $k\gg 0$, the higher iterated  Hilbert coefficients $e_j^i(A(-a,-b)_{k})$ are polynomial functions of degree $m+j-1$ with
\[
 e_j^i(A(-a,-b)_{k})\leq \binom{k-b+m-1}{m-1}\binom{p_m(k-b)+a}{j}.
\]
 Equality holds, if and only if $p_1=p_2=\cdots = p_m$ for all $j$.
\end{Proposition}

\begin{proof}
Without restriction we may assume that $p_1\leq p_2\leq \ldots \leq p_m$.

For $\beta=(\beta_1,\ldots,\beta_m)\in \ZZ^m_{\geq 0}$ and $p=(p_1,\ldots,p_m)$ we set  $|\beta|=\sum_{i=1}^m\beta_i$ and $p\beta=\sum_{i=1}^m p_i\beta_i$. Furthermore, let  $C(k-b)=\{\beta \: | \beta|=k-b\}$.

By \eqref{soonlunch}, the $S$-module $A(-a,-b)_{k}$ is the direct sum of the shifted free $S$-modules $S(-p\beta-a)$ with $\beta\in C(k-b)$.

Therefore,
\begin{eqnarray*}
P^{i}_{A(-a,-b)_{k}}(x) = \sum_{\beta \in C(k-b)} P^{i}_{S(-p\beta-a)}(x).
\end{eqnarray*}
Since $\deg P^{i}_{S(-p\beta-a)}(x)=n+i-1$ for all $\beta\in C(k-b)$, by Proposition ~\ref{noweasy} we get
\begin{eqnarray}
\label{minaishere}
e^{i}_{j}(A(-a,-b)_{k})&=& \sum_{\beta\in C(k-b)}e^{i}_{j}(S(-p\beta-a)) \\
&=&\sum_{\beta\in C(k-b)}\binom{p\beta+a}{j}.\nonumber
\end{eqnarray}

We show by induction on $m$ that  $\sum_{\beta\in C(k-b)}\binom{p\beta+a}{j}$ is a polynomial function in $k$.  In order to let the induction work we actually show more generally that  $\sum_{\beta\in C(k-b)}\binom{p\beta+\ell(k)}{j}$ is polynomial in $k$ where $\ell(k)$ is linear function of $k$.

If $m=1$, then
 \begin{eqnarray*}
 \sum_{\beta_1=0}^{k-b}\binom{p_1\beta_1+\ell(k)}{j}&=&\sum_{\beta_1=0}^{k-b}\sum_{i=0}^j\binom{p_1\beta_1}{j-i}\binom{\ell(k)}{i}\\
 &=&\sum_{i=0}^{j}\binom{\ell(k)}{i}\sum_{\beta_1=0}^{k-b}\binom{p_1\beta_1}{j-i}
 \end{eqnarray*}
 is  polynomial in $k$, because,  by (\ref{skype}),  $\sum_{\beta_1=0}^{k-b}\binom{p_1\beta_1}{j-i}$ is  polynomial in $k$. Now assume that $m>1$. We set
\[
F(k-b)=\sum_{\beta\in C(b-k)}\binom{p\beta+\ell(k)}{j}.
\]
Then
\begin{eqnarray*}
F(k-b)&=&\sum_{\beta_1=0}^{k-b}\sum_{\beta'\in C'(b-k-\beta_1)}\binom{p'\beta'+\ell(k)}{j}\\
&=&\sum_{\beta_1=0}^{k-b}F'(k-b-\beta_1)
\end{eqnarray*}
where $\beta'=(\beta_2,\ldots,\beta_m)$, $p'=(p_2-p_1,\ldots,p_m-p_1)$ , $\ell'(k)=\ell(k)+p_1(k-b)$, $C'(k-b-\beta_1)=\{\beta'\mid  |\beta'|=k-b-\beta_1\}$ and $F'(k-b-\beta_1)=\sum_{\beta'\in C'(b-k-\beta_1)}\binom{p'\beta'+\ell'(k)}{j}$.

By our  induction hypothesis,
  $F'(k-b-\beta_1)$ is polynomial in $k$.
Therefore by (\ref{skype}),   $\sum_{\beta_1=0}^{k-b}F'(k-b-\beta_1)=\sum_{i=0}^{k-b}F'(i)$ is polynomial in $k$.

These considerations together with \label{minaishere} show  that  $e^{i}_{j}(A(-a,-b)_{k}$ is polynomial function in $k$. Since
\[
\binom{k-b+m-1}{m-1}\binom{p_1(k-b)+a}{j}\leq e^{i}_{j}(A(-a,-b)_{k})\leq
\binom{k-b+m-1}{m-1}\binom{p_m(k-b)+a}{j},
\]
and since these lower and upper bounds are polynomial functions of degree $m+j-1$ with non-negative leading coefficient, we conclude that  the degree of the polynomial functions $e^{i}_{j}(A(-a,-b)_{k})$ is $m+j-1$, as well. Furthermore, it follows that  $e^{i}_{j}(A(-a,-b)_{k})=\binom{k-b+m-1}{m-1}\binom{p_m(k-b)+a}{j}$ if all $p_t$ are the same.

Conversely, since  $\binom{p\beta+a}{j}\leq \binom{p_m(k-b)+a}{j}$ for all summands $\binom{p\beta+a}{j}$ of   $e^{i}_{j}(A(-a,-b)_{k})$ it follows that $e^{i}_{j}(A(-a,-b)_{k})=\binom{k-b+m-1}{m-1}\binom{(k-b)p_m+a}{j}$ if and only if $\binom{p\beta+a}{j}= \binom{p_m(k-b)+a}{j}$ for all $\beta\in (k-b)$. In particular, if $\beta=(k-b,0,\ldots,0)$, then $\binom{p\beta+a}{j}=\binom{p_1(k-b)+a}{j}=\binom{p_m(k-b)+a}{j}$. It follows that $p_1=p_m$, and hence  $p_i=p_m$ for all  $i$.
\end{proof}

\section{The higher iterated  Hilbert coefficients of the graded components of a bigraded $A$-module}
\label{2}
Let $K$ be a field, $S=K[ x_{1},\ldots,x_{n}]$ the polynomial ring in $n$ variables with  the standard  grading, and  let as before  $A=K[ x_{1},\ldots,x_{n},y_{1},\ldots,y_{m}]$ be the polynomial ring with bigrading defined by $\deg x_{i} = (1,0)$ and $\deg y_{j}=(p_j,1)$, for some some integers $p_j\geq 0$.

Let $M$ be a finitely generated bigraded $A$-module. As before we set $M_k=\Dirsum_i M_{(i,k)}$. Then each $M_k$ is a finitely generated graded $S$-module. In this section we want to study the higher iterated   Hilbert coefficients $e_j^i(M_{k})$. We set $\mm=(x_1,\ldots, x_n)$ and $\nn=(y_1,\ldots,y_m)$. Then $A/\nn=S$ and $A/\mm=S'$ where is the polynomial ring $K[y_1,\ldots,y_m]$. Before stating the main theorem we need some preparation.

\begin{Lemma}
\label{veryeasy}
Let $M$ be a finitely generated bigraded $A$-module. Then the following holds:
\begin{enumerate}
\item[{\em (a)}] There exists an integer $s$ such that $M_{k+1}=\nn M_k$ for $k\geq s$.

\item[{\em (b)}] The Krull dimension $\dim M_k$ of $M_k$ is constant for all $k\gg 0$.

We set $\ldim M=\lim_{k\to \infty} \dim M_k$.
\item[{\em (c)}] Let $M'=\Dirsum_{k\geq k_0} M_k$ where $k_0$ is chosen such that $\dim M_k= \ldim M$ and  $M_{k+1}=\nn M_k$ for all $k\geq k_0$. Then
\begin{enumerate}
\item[{\em (i)}] $\dim M'/\nn M'= \ldim M'=\ldim M$;
\item[{\em (ii)}] $\dim M'/\mm M'= \dim M/\mm M$.
\end{enumerate}
\end{enumerate}
\end{Lemma}
\begin{proof}
(a) Set $N=M/\nn M$. Then $N$ is a finitely generated bigraded $A$-module with $\nn N=0$.  Let $n_1,\ldots,n_r$ be a set of bihomogeneous generators of $N$ with $\deg n_i=(l_i,k_i)$, and let $s$ be the maximum of the $k_i$. We claim that $N_k=0$ for $k>s$. Indeed, let $u\in N_k$.  We may assume that $u$ is bihomogeneous, say $\deg u=(j,k)$. Then there exist bihomogeneous elements $f_1,\ldots,f_r$ with $\deg f_i=(a_i,b_i)$ and $u=\sum_{i=1}^rf_in_i$ such that $b_i+k_i=k$. It follows that $b_i>0$ for all $i$. Therefore each monomial in the support of $f_i$ contains as a factor a monomial in the $y_j$ of degree $b_i$. Since all $y_j$ annihilate each  $n_i$, we see that $u=0$, and hence $N_k=0$. Consequently,   $M_{k+1}=\nn M_k$ for $k\geq s$.

(b) Let  $s$  be as in (a). Then  $M_{k+1}=\nn M_k$ for all $k\geq s$.  So $\Ann_S M_k \subseteq\Ann_S M_{k+1}$ for all $k\geq s$. Since $S$ is Noetherian, there exists $k_0\geq s$ such that  $\Ann M_k =\Ann M_{k+1}$  for all $k\geq k_0$. Then
$\dim M_{k+1}=\dim M_k $ for all $k\geq k_0$.

(c)(i) Since $M'/\nn M'\iso M_{k_0}$ it follows that $\dim M'/\nn M'=\dim M_{k_0}=\ldim M$.

(c)(ii)  $M'/\mm M'= \Dirsum_{k\geq k_0} M_k/\mm M_k$ and  $M/\mm M= \Dirsum_{k} M_k/\mm M_k$. Therefore, $M'/\mm M'$ is an $S'$-submodule of  $M/\mm M$ and
\[
(M/\mm M)/(M'/\mm M')=\Dirsum_{k<k_0}  M_k/\mm M_k.
\]
Since there are only finitely many $k<k_0$ with $M_k\neq 0$, it  follows that $$\dim (M/\mm M)/(M'/\mm M')=0.$$ This implies that $\dim M/\mm M=\dim M'/\mm M'$, as desired.
\end{proof}

In the following we use the convention that the zero polynomial has degree $-1$.

\begin{Theorem}
\label{now main}
Let $M$ be a finitely generated bigraded $A$-module. Then for $k\gg 0$, $ e_j^i(M_{k})$  is a polynomial in $k$,   and
\[
\deg  e_j^i(M_{k}) \leq m+j-1 \quad \text{for}\quad  j=0,\ldots, \ldim M+i-1,
\]
and $e_j^i(M_{k}) =0$ for $j>  \ldim M+i-1$.
\end{Theorem}
\begin{proof}
Let $M'$ be defined as in Lemma~\ref{veryeasy}. Since $M'_k=M_k$ for $k\gg 0$, we have that  $e_j^i(M'_{k})=e_j^i(M_{k})$ for $k\gg 0$. Therefore, since $\dim M'/\mm M'=\dim M/\mm M$,  we may replace $M$ by $M'$, and hence may assume from the very beginning that $M$ itself satisfies condition (c)(i) and (c)(ii) of Lemma~\ref{veryeasy}.

Let $J= \Ann_S(M/\nn M)$. Then  $M/\nn M$ is a finitely generated module over the standard graded $K$-algebra $B=S/J$. We may assume that $K$ is infinite, because otherwise we may  replace $K$ by a suitable base field extension.  By the graded Noether Normalization Theorem (see \cite[Theorem 1.5.17]{BH}), there exist linear forms  $z_1,\ldots,z_d \in S$  such that  $B$ is a finitely generated $K[z_1,\ldots,z_d]$-module, where $d= \dim M/\nn M$. It follows that $M$ is a finitely generated bigraded $A'$-module, where $A'=K[z_1,\ldots,z_d, y_1,\ldots,y_m]\subset A$ with $\deg z_i=(1,0)$. Indeed, since $M/\nn M$ is a finitely generated bigraded $B$-module, and $B$ is a finitely generated
bigraded $A'/\nn A'=K[z_1,\ldots,z_d]$ -module, it follows that $M/\nn M$ is a finitely generated bigraded $A'/\nn A'$-module. Therefore, by  Nakayama's Lemma,  $M$ is a finitely generated bigraded  $A'$-module.

Now let $\FF$ be a bigraded minimal free $A'$-resolution of $M$ with
\[
F_r=\Dirsum_sA'(-a_{rs},-b_{rs})\quad \text{for all} \quad  r.
\]
Then $\FF_k$ is a graded free $K[z_1,\ldots,z_d]$-resolution  of $M_k$, where $\FF_k$ is the $k$th graded piece of $\FF$ which is obtained from $\FF$ by restricting the differentials of $\FF$ to the graded components $(F_r)_k=\Dirsum_j (F_r)_{(j,k)}$.
It follows that $P^i_{M_k}(x)=\sum_r (-1)^{r+1}P^i_{(F_r)_k}(x)$. Since each $(F_r)_k$ is a free $S$-module, each  $(F_r)_k$ has dimension $n$.

We write $\sum_r (-1)^{r+1}P^i_{(F_r)_k}(x)=\sum_j(-1)^jf^i_{j,k}\binom{x+d-j}{d-j}$. Since the coefficients  $f^i_{j,k}$ are linear combination of terms of the form $e^i_j(A'(-a,-b)_k$), it follows from Proposition~\ref{rain} that  the  $f^i_{j,k}$ are polynomials in $k$ of degree $\leq m+j-1$.

Since by  Lemma~\ref{veryeasy}, $\dim M_k=d$ for all $k\gg 0$, we see that  $e^i_j(M_k)=f^i_{j,k}$ for all $i$,$j$ and $k\gg 0$. This yields the desired result.
\end{proof}

\begin{Remark}
The fact that  $\deg e^i_j$ independent on $i$ is also consequence of Corollary~\ref{reject}.
\end{Remark}

\begin{Example}
Let $S=K[x_1,x_2]  ,  \mm =(x_1,x_2)$ and $\MR(\mm)=\Dirsum_{k\geq 0}\mm^k$. Let $A=K[x_1,x_2,y_1,y_2]$ with bigrading
defined by $\deg(x_i) = (1,0)$ and $\deg(y_i) = (1,1)$, for $i=1,2$. The natural map defined by $x_i \mapsto x_i$ and $y_i\mapsto x_it$, for $i=1,2$, is then a surjective homomorphism. So $\MR(\mm)$ has a bigraded free resolution of the form
\[
 0 \rightarrow A(-2,-1) \rightarrow A \rightarrow \MR(\mm) \rightarrow 0
\]
Hence $e^i_j(\mm^k)=e^i_j(A_k)-e^i_j(A(-2,-1)_k)$. One has $e^i_j(A_k)=(k+1)\binom{k}{j}$ and $e^i_j(A(-2,-1)_k)=k\binom{k+1}{j}$. So $e^i_j(\mm^k)=\frac{k(k-1)\cdots (k-j+2)}{j}(1-j)(k+1)$. Therefore $\deg(e^i_j(\mm^k))=j$ , and by Theorem ~\ref{now main} our upper bound is $j+1$.
\end{Example}

In the special case that  all $p_i$ are the same, we can improve the  upper bound for the degree of the higher iterated  Hilbert coefficients as follows:

\begin{Theorem}
\label{main}
Assume that  $p_1=p_2=\cdots = p_m=p$, and let $M$ be a finitely generated bigraded $A$-module. Then for $k\gg 0$,
$ e_j^i(M_{k})$  is a polynomial in $k$,   and
\[
\deg  e_j^i(M_{k}) \leq \dim M/\mm M+j-1 \quad \text{for}\quad  j=0,\ldots, \ldim M+i-1,
\]
and $e_j^i(M_{k}) =0$ for $j>  \ldim M+i-1$.
\end{Theorem}

\begin{proof}
By using the Noether Normalization Theorem, we may replace, as in the proof of Theorem~\ref{now main}, $A$ by $A''=K[z_1,\ldots,z_d,w_1,\ldots,w_{d'}]$ where $d=\ldim M$ and $d'=\dim M/\mm M$. Then by computing the higher iterated Hilbert polynomial by using a bigraded free $A''$-resolution of $M$,  yields, as in the proof of Theorem~\ref{now main},  the desired conclusion.
\end{proof}

The given upper bound for the degree of the higher iterared Hilbert coefficients  of a bigraded $A$-module as given in Theorem~\ref{main} is in general sharp, for example for $M=A$. In more special cases it may not be sharp. Indeed, let
$I\subset S$ be a graded ideal generated by $m$ homogeneous polynomials of  degree $p$, and let $\MR(I)=\Dirsum_{k\geq 0}I^k$ the Rees ring of $I$. Then $\MR(I)$ is a bigraded $A$-algebra with $\MR(I)_k\iso I^k$ and $\dim \MR(I)/\mm \MR(I)=\ell(I)$, which by definition is the analytic spread of $I$. Thus we have

\begin{Corollary}
\label{analytic}
Let $I\subset S$ be a graded ideal generated in a single degree. Then  for all $k\gg 0$, $e^i_j(I^k)$ is a polynomial function of degree $\leq \ell(I)+j-1$.
\end{Corollary}

In  case that $I$ is $\mm$-primary, one has  $e^i_0(I^k)=1$ for all $i$ and $k$ so that $\deg e^i_0(I^k)=0$, while the formula in Corollary ~\ref{analytic} gives the degree bound $n-1$, since $\ell(I)=n$.

\section{The higher iterated  Hilbert coefficients of the graded components of Tor and Ext}
\label{3}
Let $M$ be a graded $S$-module and $N=\Dirsum_{i,j\in \NN}N_{(i,j)}$  bigraded $A$-module. We will see that $\Tor_i^S(M,N)$ and $\Ext^i_S(M,N)$ are naturally bigraded $A$-modules. Thus we may then study the higher iterated  Hilbert coefficients of the graded components of these modules.

\medskip
Let $U$ by a finitely generated graded $S$-module, and $V$ be a finitely generated  bigraded $A$-module. We first notice that
\[
U\tensor_S V \quad \text{and}\quad \Hom_S(U,V)
\]
are bigraded $A$-modules. Indeed,
\[
(U\tensor_S V)_{(c,d)}=\Dirsum_k U_k\tensor_K V_{(c-k,d)},
\]
and
\[
\Hom_S(U,V)_{(c,d)} =\{f\in  \Hom_S(U,V)\:\; f(U_i)\subset V_{(i+c,d)} \text{ for all } i \}.
\]
With this bigraded structure as described above we have
\[
(U\tensor_S V)_k=U\tensor_S V_k \quad \text{and}\quad  \Hom_S(U,V)_k=\Hom_S(U,V_k) \quad \text{for all}\quad k.
\]

\begin{Lemma}
\label{torext}
Let $M$ be a finitely generated graded $S$-module and $N$  finitely generated bigraded $A$-module. Then, for all $i$,  $\Tor^S_i(M,N)$ and $\Ext^i_S(M,N)$ are finitely generated bigraded $A$-modules, and
\[
\Tor^S_i(M,N)_k\iso \Tor^S_i(M,N_k) \quad \text{and}\quad \Ext_S^i(M,N)_k\iso \Ext_S^i(M,N_k)\quad \text{for all $i$ and $k$}.
\]
\end{Lemma}

\begin{proof}
Let $\FF$ bigraded free $A$-resolution of $N$. Then
\[
\Tor_i^S(M,N)_k=H_i(M\tensor_S\FF)_k=H_i((M\tensor_S\FF)_k)=H_i(M\tensor_S\FF_k)=\Tor_i^S(M,N_k).
\]
Here we used that taking the graded components can be exchanged with taking homology, and we also used that $\FF_k$ is a graded free $S$-resolution of $N_k$.

In order to compute $\Ext^i_S(M,N)$ we choose a graded free $S$-resolution $\FF$ for $M$. Then $\Hom_S(\FF, N)$ is a complex of bigraded $A$-modules, and $\Ext^i_S(M,N)=H^i(\Hom_S(\FF, N))$ has a natural bigraded structure. Moreover,
\begin{eqnarray*}
\Ext^i_S(M,N)_k&=& H^i(\Hom_S(\FF, N))_k=H^i(\Hom_S(\FF, N)_k)=H^i(\Hom_S(\FF, N_k))\\
&=&\Ext^i_S(M,N_k) \quad \text{for all $k$}.
\end{eqnarray*}
\end{proof}

As a consequence of Lemma~\ref{veryeasy} we obtain

\begin{Corollary}
\label{dimension}
Let $M$ be a finitely generated graded $S$-module and $N$  finitely generated bigraded $A$-module. Then
the Krull dimension of the finitely generated graded $S$-modules  $\Tor_i^S(M,N)_k$ and $\Ext_S^i(M,N)_k$ are  constant for $k\gg 0$.
\end{Corollary}

Next we want to study further the graded $S$-modules $\Tor_i^S(M,N)_k$ and $\Ext_S^i(M,N)_k$. By the preceding corollary, their Hilbert polynomials have constant degree for large $k$. For $\Tor_i^S(M,N)_k$, these degrees can be bounded as follows

\begin{Proposition}
With the notation and assumptions as before, we have
\[
\dim \Tor_{i+1}^S(M,N)_k\leq  \dim \Tor_i^S(M,N)_k \quad \text{for all $k$}.
\]
In particular, $\dim \Tor_{i+1}^S(M,N)_k\leq \dim (M\tensor_S N_k)$ for all $k$, and hence for $k\gg  0$, the degree of the $j$th iterated Hilbert polynomial of $\Tor_i^S(M,N)_k$ is less than or equal to $\dim(M\tensor_S \ldim N)+j-1$.
\end{Proposition}
\begin{proof}
Let $T=K[y_1,\ldots,y_n]$. We may view $N_k$ a graded $T$-module by setting $y_iu:=x_iu$ for all $i=1,\ldots,n$ and $u\in N_k$.  So $M\tensor_K N_k$ has the natural structure of an  $S\tensor_K T$-module, and
\[
 \Tor_i^S(M,N)_k\iso \Tor_i^S(M,N_k)\iso H_i(x_1-y_1,\ldots,x_n-y_n; M\tensor_KN_k),
\]
where $H_i(\_)$ denotes Koszul homology. (see \cite[Chapter IX,Theorem 2.8]{CE} and \cite[page 101]{S})

Thus in order to see that  $\dim \Tor_{i+1}^S(M,N)_k  \leq \dim \Tor_i^S(M,N)_k$ it suffices to show that
whenever $W$ is a graded module over a polynomial ring $R$, and $\xb$ is a finite sequence of elements of $R$, then $\dim H_{i+1}(\xb;R)\leq \dim H_i(\xb;W)$ for all $i$. To see this,  let $P$ be in the support of  $H_{i+1}(\xb;W)$. Then we have to show that $P$ is in the support of $H_i(\xb;W)$. Since $H_{i+1}(\xb;W_P)=H_{i+1}(\xb;W)_P\neq 0$ it follows from \cite[Exercise 1.6.31]{BH} that $H_i(\xb;W_P)=H_i(\xb;W)_P\neq 0$, and the desired conclusion follows.
\end{proof}

\begin{Corollary}
\label{alwaysalabel}
Let $M$ be a finitely generated graded $S$-module and $N$  finitely generated bigraded $A$-module. Then for all $k\gg 0$,  $e_j^i(\Tor^S_l(M,N_k))$  and $e_j^i(\Ext^l_S(M,N_k))$ are  polynomials in $k$ of degree at most $m-1+j$

 In special the case that $p_i=p$ for all $i$, the degree of $e^i_j(\Tor^S_l(M,N_k))$ is bounded by  $\dim \Tor^S_l(M,N)/\mm \Tor^S_l(M,N)+j-1$  and the degree of $e^i_j(\Ext^l_S(M,N_k))$ is bounded by $\dim \Ext^l_S(M,N)/\mm \Ext^l_S(M,N)$.
\end{Corollary}

\begin{Corollary}
\label{ik}
Let $M$ be a graded $S$-module, and $I\subset S$ a graded ideal. Then for $l>1$, $e^i_j(\Tor_l^S(M,S/I^k))$ is polynomial function in $k$ of degree less than or equal to $v(I)+j-1$ where $v(I)$ denotes the number of generators of $I$. If all generators of $I$ have the same degree then $v(I)$ can be replaced by  $\dim \MR(I)/\Ann_S(M)\MR(I)$.
\end{Corollary}
\begin{proof}

The exact sequence
\[
 0 \rightarrow I^k \rightarrow S \rightarrow S/I^k \rightarrow 0,
\]
implies that $\Tor^S_l(M,S/I^k)\iso \Tor^S_{l-1}(M,I^k)\iso \Tor^S_{l-1}(M,\MR(I))_k$ for all $l> 1$ where $\MR(I))$ is the Rees ring of $I$.  So,  by Corollary~\ref{alwaysalabel}, we see that   $e^i_j(\Tor^S_l(M,S/I^k))$
is a polynomial in $k$  for all $k\gg 0$ of degree less than or equal to $v(I)+j-1$. In the special case that all generators of $I$ have the same degree, Corollary~\ref{alwaysalabel} implies that
\begin{eqnarray*}
\deg(e^i_j(\Tor^S_l(M,S/I^k))&\leq& \dim \Tor^S_{l-1}(M,I^k)+j-1 \leq \dim(M \tensor_S \MR(I))+j-1\\ \nonumber
&\leq& \dim (S/\Ann_S(M)\tensor_S \MR(I))+j-1\\  \nonumber
 &\leq &\dim \MR(I)/(\Ann_S(M)\MR(I))+j-1.\nonumber
\end{eqnarray*}
\end{proof}


\begin{thebibliography}{}

\bibitem{BH} W.~Bruns, and J.~Herzog, Cohen--Macaulay rings, Revised Edition, Cambridge University Press, Cambridge, 1996.

\bibitem{CE} H.~Cartan and S.~Eilenberg, Homological algebra, Oxford University Press, Cambridge, 1973.

\bibitem{CKST} A.~Crabbe, D.~Katz, J.~Striuli and E.~Theodorescu, Hilbert-Samuel polynomials for the contravariant extension functor,{\em Nagoya Math. J.} {\bf 198} (2010), 1--22.

\bibitem{HPV} J.~Herzog, T.J.~Puthenpurakal and J.K.~Verma, Hilbert polynomials and powers of ideals, {\em   Math. Proc. Camb. Phil. Soc.} {\bf 145} (2008), 623--642.

\bibitem{HW} J.~Herzog and V.~Welker, The Betti polynomials of powers of an ideal, {\em J. Pure Appl. Algebra} {\bf 215} (2011), 589--596.

\bibitem{K1} V.~Kodiyalam, Homological invariants of powers of an ideal, {\em Proc. Amer. Math. Soc.} {\bf 118} (3) (1993), 757--764.

\bibitem{K2} V.~Kodiyalam, Asymptotic behaviour of Castelnuovo-Mumford regularity, {\em Proc. Amer. Math. Soc.} {\bf 128} (2) (1999), 407--411.

\bibitem{KT1} D.~Katz and E.~Theodorescu, On the degree of Hilbert polynomials associated to the torsion functor, {\em Proc. Amer. Math. Soc.} {\bf 135} (10) (2007), 3073--3082.

\bibitem{KT2} D.~Katz and E.~Theodorescu, Hilbert polynomials for the extension functor, {\em J. Of Algebra} {\bf 319} (2008), 2319--2336.

\bibitem{S} J.P.~Serre, Local algebra, Springer Monographs in Mathematics, Springer, 2000.

\bibitem{T} C.~Theodorescu, Derived functors and Hilbert polynomials, {\em Math. Proc. Camb. Phil. Soc.} {\bf 132} (01) (2002), 75--88.

\end{thebibliography}
\end{document}